\documentclass[a4paper,12pt]{article}
\usepackage{amsmath,amsfonts,amssymb,amsthm}
\usepackage{amssymb}
\usepackage{bbm}
\usepackage[dvips]{graphicx}
\usepackage{pdfsync}
\usepackage{color,xcolor}
\usepackage{verbatim}
\usepackage{enumerate}
\usepackage{authblk}

\newcommand{\abs}[1]{\left\vert#1\right\vert}
\newcommand{\set}[1]{\left\{#1\right\}}
\newcommand{\eps}{\varepsilon}

\newcommand{\CC}{\mathcal{C}}

\newcommand{\CT}{\mathcal{T}}

\newcommand{\CS}{\mathcal{S}}

\renewcommand{\emptyset}{\varnothing}

\newcommand{\sd}{\bigtriangleup}

\renewcommand{\epsilon}{\varepsilon}

\renewcommand{\geq}{\geqslant}

\newtheorem{thm}{Theorem}[section]
\newtheorem*{thm*}{Theorem}
\newtheorem{lem}[thm]{Lemma}
\newtheorem*{lem*}{Lemma}
\newtheorem{cor}[thm]{Corollary}
\newtheorem*{cor*}{Corollary}

\newtheorem*{prop*}{Proposition}
\theoremstyle{definition}
\newtheorem{defn}[thm]{Definition}
\newtheorem*{defn*}{Definition}
\theoremstyle{remark}

\newtheorem*{rem*}{Remark}

\newtheorem*{example*}{Example}

\newtheorem*{que*}{Question}

\title{Dynamical quasitilings of amenable groups}

\author{Tomasz Downarowicz, Dawid Huczek}
\affil{\textit{Faculty of Pure and Applied Mathematics,
        Wroc{\l}aw University of Science and Technology, 
        Wybrze\.{z}e Wyspia\'{n}skiego 27,
        50-370 Wroc{\l}aw, Poland}\\
    \texttt{Tomasz.Downarowicz@pwr.edu.pl, Dawid.Huczek@pwr.edu.pl}}

\begin{document}

\maketitle

\begin{abstract}
We prove that for any compact zero-dimensional metric space $X$ on which an infinite countable amenable group $G$ acts freely by homeomorphisms, there exists a dynamical quasitiling with good covering, continuity, F\o lner and dynamical properties, i.e to every $x\in X$ we can assign a quasitiling $\mathcal{T}_x$ of $G$ (with all the $\mathcal{T}_x$ using the same, finite set of shapes) such that the tiles of $\mathcal{T}_x$ are disjoint, their union has arbitrarily high lower Banach Density, all the shapes of $\mathcal{T}_x$ are large subsets of an arbitrarily large F\o lner set, and if we consider $\mathcal{T}_x$ to be an element of a shift space over a certain finite alphabet, then the mapping $x\mapsto \mathcal{T}_x$ is a factor map.
%\keywords{countable amenable group, (dynamical) quasitiling}
\end{abstract}

\section{Introduction}
Many constructions in symbolic and zero-dimensional dynamics with the action of $\mathbb{Z}$ rely on partitioning a sequence of symbols into disjoint blocks (i.e. partitioning $\mathbb{Z}$ into disjoint intervals) and performing some operations on such blocks. Analogous constructions in systems with the action of an arbitrary amenable group $G$ require partitioning $G$ into disjoint, finite subsets with good invariance properties --- in other words, constructing an appropriate (quasi-)tiling of $G$. Such quasitilings, in which tiles are ,,almost'' disjoint and whose union is ,,almost'' all of $G$ were first developed by Ornstein and Weiss in \cite{OW}, and have later been improved to tilings in \cite{DHZ}. 

The (quasi-)tilings of \cite{OW} and \cite{DHZ} are algebraic in nature and their construction does not rely on any ,,external'' dynamics. However, proofs in traditional symbolic and zero-dimensional dynamics often require that the partitioning of symbolic sequences into blocks be continuous and consistent with the shift actions, i.e. that two sequences that agree on a long interval should also be identically partitioned over this (or slightly shorter) interval, and that shifting a sequence should result in shifting the corresponding partition. In the present paper we prove that given a free action of a countable, discrete, amenable group $G$ on a zero-dimensional compact metric space $X$, it is possible to construct a quasitiling $\CT_x$ for every $x\in X$ such that all such quasitilings have arbitrarily good invariance and covering properties, and if we view $\CT_x$ as an element of an appropriate shift space, then the mapping $x\mapsto \CT_x$ is a factor map.
A careful reader will notice that large portions of this note are copied from \cite{DHZ}. We do so in order to make this paper self-contained.

\section{Preliminaries}
\subsection{Basic notions}
Throughout this paper $G$ denotes an infinite countable \emph{amenable group}, i.e., a group in which there exists a sequence of finite sets $F_n\subset G$ (called a \emph{F\o lner sequence}, or a \emph{sequence of F\o lner sets}), such that for any $g\in G$ we have
\[
\lim_{n\to\infty}\frac{\abs{gF_n\sd F_n}}{\abs{F_n}}= 0,
\]
where $gF=\set{gf:f\in F}$, $\abs{\cdot}$ denotes the cardinality of a set, and $\sd$ is the symmetric difference. Since $G$ is infinite, the sequence $|F_n|$ tends to infinity. Without loss of generality (see \cite[Corollary 5.3]{N}) we can assume that the sets in the F\o lner sequence are symmetric (i.e. $F_n^{-1}=F_n$ for every $n$) and contain the unit.

\begin{defn}
If $T$ and $K$ are nonempty, finite subsets of $G$ and $\eps<1$, we say that $T$ is \emph{$(K,\eps)$-invariant} if
\[\frac{\abs{KT\sd T}}{\abs{T}}< \eps,\]
where $KT=\set{gh : g \in K,h\in T}$.
\end{defn}

Observe that if $K$ contains the unit of $G$, then $(K,\eps)$-invariance is equivalent to the simpler condition
\[\abs{KT}< (1+\eps)\abs{T}.\]

The following facts is are not difficult to see, so we skip the proofs, referring the reader to \cite{DHZ}.
\begin{lem}\label{folner}
A sequence of finite sets $(F_n)$ is a F\o lner sequence if and only if for every finite set $K$ and every $\eps>0$ the sets $F_n$ are eventually $(K,\eps)$-invariant.
\end{lem}

\begin{lem}\label{ben}
Let $K\subset G$ be a finite set and fix some $\eps>0$. There exists $\delta>0$ such that if
$T\subset G$ is $(K,\delta)$-invariant and $T'$ satisfies $\frac{|T'\triangle T|}{|T|}\le \delta$ then
$T'$ is $(K,\eps)$-invariant.
\end{lem}
%\begin{proof}
%We have $KT'\setminus KT\subset K(T'\setminus T)$ (and similarly for $T$ and $T'$ exchanged), so
%$KT'\triangle KT= (KT'\setminus KT)\cup(KT\setminus KT')\subset K(T'\setminus T)\cup K(T\setminus T')=
%K(T'\triangle T)$. Thus, by the triangle inequality for $|\cdot\triangle\cdot|$, we obtain
%$$
%\frac{|KT'\triangle T'|}{|T'|}\le
%\frac{|KT'\triangle KT|+ |KT\triangle T|+|T'\triangle T|}{(1-\delta)|T|}\le  \frac{|K|\delta+\delta+\delta}{1-\delta}<\eps,
%$$
%if $\delta$ is sufficiently small.
%\end{proof}

\begin{defn}
We say that $T'\subset T$ ($T$ finite) is a $(1-\eps)$-subset of $T$ if $\abs{T'}\geq(1-\eps)\abs{T}$.
\end{defn}

%\begin{defn}
%Let $K$ be a finite subset of $G$ and let $T\subset G$ be arbitrary. The \emph{$K$-core} of $T$, denoted as $T_K$, is the set $\set{g\in T: Kg\subset T}$ (this is the largest subset $T'\subset T$ satisfying $KT'\subset T$).
%\end{defn}

%\begin{lem}\label{estim}
%For any $\eps>0$ and any finite $K\subset G$ there exists a $\delta$ (in fact $\delta = \frac\eps{|K|}$), such that if $T\subset G$ is finite and $(K,\delta)$-invariant then the $K$-core $T_K$ is a $(1-\eps)$-subset of $T$.
%\end{lem}
%
%\begin{proof} Note that $(K,\delta)$-invariance of $T$ implies that
%$$(\forall g\in K)\ \ |gT\setminus T|<\delta|T|,$$
%i.e.,
%$$(\forall g\in K)\ \ |T\cap g^{-1}T| = |gT\cap T| >(1-\delta)|T|.
%$$
%This yields
%$$|T_K|=\left|\bigcap_{g\in K}(T\cap g^{-1}T)\right|>(1-|K|\delta)|T|=(1-\eps)|T|.$$
%\end{proof}
%

Like in \cite{DHZ}, we will use the following definition of lower Banach density:

\begin{defn}
For $S\subset G$ and a finite, nonempty $F\subset G$ denote
\[
\underline D_F(S)=\inf_{g\in G} \frac{|S\cap Fg|}{|F|}. 
\]
If $(F_n)$ is a F\o lner sequence then define
\[
\underline D(S)=\limsup_{n\to\infty} \underline D_{F_n}(S),
\]
which we call the \emph{lower Banach density} of $S$.
\end{defn}
The proof of the following standard fact can again be found in \cite{DHZ}:
\begin{lem}\label{bd}
Regardless of the set $S$, the value of $\underline D(S)$ does not depend on the
F\o lner sequence, the limit superior in the definition is fact a limit, and moreover
\[\underline D(S) = \sup\{\underline D_F(S): F\subset G, F \text{ is finite}\}\]
\end{lem}
%
%\begin{proof}
%To prove the assertion, it suffices to show that
%$$
%\liminf_n\underline D_{F_n}(S)\ge\sup\{\underline D_F(S): F\subset G, F \text{ is finite}\}.
%$$
%Fix some $\eps>0$ and let $F$ be a finite set such that
%$$
%\underline D_F(S)\ge \sup\{\underline D_{F'}(S): F'\subset G, F' \text{ is finite}\}-\eps.
%$$
%Let $n$ be so large that $F_n$ is $(F,\eps)$-invariant. Given $g\in G$, we have
%$$
%|S\cap Ffg|\ge \underline D_F(S)|F|,
%$$
%for every $f\in F_n$. This implies that there are at least $\underline D_F(S)|F||F_n|$ pairs $(f',f)$ with $f'\in F, f\in F_n$  such that
%$f'fg\in S$. This in turn implies that there exists at least one $f'\in F$ for which
%$$
%|S\cap f'F_ng|\ge \underline D_F(S)|F_n|.
%$$
%Since $f'\in F$ and $F_n$ is $(F,\eps)$-invariant (and hence so is $F_ng$), we have
%$$
%|S\cap f'F_ng|\le|S\cap FF_ng|\le |S\cap F_ng|+\eps|F_n|,
%$$
%which yields
%$$
%|S\cap F_ng|\ge (\underline D_F(S)-\eps)|F_n|.
%$$
%We have proved that $\underline D_{F_n}(S)\ge \underline D_F(S)-\eps$, which ends the proof.
%\end{proof}

Let $X$ be a compact metric space and let $G$ be a discrete amenable group. We say that $G$ acts on $X$ by homeomorphisms, if for every $g\in G$ there exists a homeomorphism $h_g:X\mapsto X$ such that the mapping $h\mapsto h_g$ is a group isomorphism between $g$ and a subgroup of $\text{Homeo}(X)$. In a slight abuse of notation, $h_g(x)$ is often written just as $g(x)$, i.e. $G$ is identified with a subgroup of $\text{Homeo}(X)$. We say that the action of $G$ is free, if the identity $g(x)=x$ for some $x\in X,g\in G$ implies that $g=e$.

Let $\Lambda$ be a finite set with the discrete topology. There exists a standard action of $G$ on $\Lambda^G$ (called the \emph{shift} action), defined as follows: $(gx)(h)=x(hg)$. $\Lambda^G$ with the product topology and the shift action of $G$ becomes a zero-dimensional dynamical system, called the \emph{full shift over $\Lambda$}. A \emph{symbolic dynamical system over $\Lambda$} is any closed, $G$-invariant subset $X$ of the full shift.

\section{Quasitilings}

The first two definitions below are the same as in \cite{DHZ}
\begin{defn} A \emph{quasitiling} is determined by two objects:
\begin{enumerate}
	\item a finite collection $\CS(\CT)$ of finite subsets of $G$ containing the unit $e$,
	called \emph{the shapes}.
	\item a finite collection $\CC(\CT) = \{C(S):S\in\CS(\CT)\}$ of disjoint subsets of $G$, called \emph{center sets} (for the shapes).
\end{enumerate}
The quasitiling is then the family $\CT=\{(S,c):S\in\CS(\CT),c\in C(S)\}$. We require that the map $(S,c)\mapsto Sc$ be injective.\footnote{This requirement is stronger than asking that different tiles have different centers. Two tiles $Sc$ and $S'c'$ may be equal even though $c\neq c'$
(this is even possible when $S=S'$). However, when the tiles are disjoint, then the (stronger) requirement
follows automatically from the fact that the centers belong to the tiles.} Hence, by the \emph{tiles} of $\CT$ (denoted by the letter $T$) we will mean either the sets $Sc$ or the pairs $(S,c)$ (i.e., the tiles with defined centers), depending on the context. 
\end{defn}
Note that every quasitiling $\CT$ can be represented in a symbolic form, as a point $x_\CT\in\Lambda^G$, with the alphabet $\Lambda = \CS(\CT)\cup\{\emptyset\}$, as follows: $x_\CT(g)=S \iff g\in C(S)$, $\emptyset$ otherwise.

\begin{defn}\label{qt} Let $\eps\in[0,1)$ and $\alpha\in(0,1]$. A quasitiling $\CT$ is called
\begin{enumerate}
	\item \emph{$\eps$-disjoint} if there exists a mapping $T\mapsto T^\circ$ ($T\in\CT$) such that
	\begin{itemize}
	\item $T^\circ$ is a $(1-\eps)$-subset of $T$, and
	\item $T\neq T'\implies T^\circ\cap {T'}^\circ=\emptyset$;
	\end{itemize}
    \item \emph{disjoint} if the tiles of $\CT$ are pairwise disjoint;
	\item \emph{$\alpha$-covering} if $\underline D(\bigcup\CT)\ge\alpha$.
\end{enumerate}
\end{defn}

\begin{defn}
 If $X$ is a zero-dimensional compact metric space and $G$ is an amenable group acting on $X$ by homeomorphisms, then a \emph{dynamical quastiling} is a map $x\mapsto \CT_x$ which assigns to every $x\in X$ a quasitiling of $G$ such that the set of all shapes $\CS=\bigcup_{x\in X}\CS(\CT_x)$ is finite, and $x\mapsto \CT_x$ is a factor map from $X$ onto a symbolic dynamical system over the alphabet $\Lambda=\CS\cup\set{\emptyset}$. We say that a dynamical quastitiling is $\eps$-disjoint, disjoint or $\alpha$-covering, if $\CT_x$ has the respective property for every $x$.
\end{defn}

The following lemma is very similar to the one in \cite{DHZ} (which in turn largely uses the techniques developed in \cite[I.\S 2. Theorem 6]{OW}), with the major difference concerning the dynamical properties of the obtained tilings.

\begin{lem}\label{quasitilings} Let $X$ be a compact, zero-dimensional metric space and let $G$ be a countable amenable group acting freely on $X$ by homeomorphisms, with a F\o lner sequence $(F_n)$ of symmetric sets containing the unit. Given $\eps>0$, there exists a positive integer $r=r(\eps)$ such that for each positive integer $n_0$ there exists a dynamical quastitiling $x\mapsto \CT_x$ which is $\eps$-disjoint and $(1-\eps)$-covering, and the set of shapes $\bigcup_{x\in X}\CS(\CT_x)$ consists of $r$ shapes $\{F_{n_1},\dots,F_{n_r}\}$, where $n_0<n_1<\cdots<n_r$.
\end{lem}

\begin{proof}
Find $r$ such that $(1-\frac\eps2)^r<\eps$. This is going to be the cardinality of the family of shapes.
Choose integers $n_1=n_0+1, n_2,\dots, n_r$ so that they increase and for each pair of indices $j<i$, $j,i\in\{1,2,\dots, r\}$ the set $F_{n_i}$ is $(F_{n_j},\delta_j)$-invariant, where $\delta_j$ will be specified later. For every $x$, we let $\CS(\CT_x) = \{F_{n_j}: j=1,\dots,r\}$ be our family of shapes. With this choice, the assertions about the shapes and their number are fulfilled. It remains to construct the corresponding center sets $C_x(F_{n_j})$ for every $x$ so as to satisfy $\eps$-disjointness and $(1-\eps)$-covering of $\CT_x$, and to ensure that the mapping $x\mapsto \CT_x$ is a factor map (to which end $\CT_x$ must depend continuously on $X$, and for eveyry $g\in G$ we must have $\CT_{gx}=g(\CT_x)$).

We proceed by induction over $j$ decreasing from $r$ to 1. Begin with $j=r$. Since $G$ acts freely on $X$, every $x\in X$ has a clopen neighborhood $U_x$ such that the sets $g(U_x)$ are pairwise disjoint for $g\in F_{n_r}$. By compactness, we can choose a finite number of such neighborhoods whose union covers $X$. We will label these neighborhoods $U_1,U_2,\ldots,U_m$.
 
 Fix $x\in X$, and let 
 \[C^{(r)}_1(x)=\set{g\in G:gx\in U_1}.\]
 Then for $i=2,\ldots,m$ let
 \[C_i^{(r)}(x)=C_{i-1}^{(r)}(x)\cup\set{g\in G: gx\in U_i \text{, and } \abs{F_{n_r}g\cap F_{n_r}C^{(r)}_{i-1}(x)}<\eps\abs{F_{n_r}}}. \]
 We can show that for every $x\in X,h\in G$ and $i=1,\ldots,m$, $C_i^{(r)}(hx)=C^{(r)}_i(x)h^{-1}$. Indeed, for $i=1$ we have $g\in C^{(r)}_1(hx)$ if and only if $ghx\in U_1$, which is equivalent to stating that $gh\in C^{(r)}_1(x)$, i.e. $g\in C_1^{(r)}(x)h^{-1}$. Now assuming that $C^{(r)}_{i-1}(hx)=C^{(r)}_{i-1}(x)h$, we have the following equivalent statements:
 \[g\in C^{(r)}_i(hx)\]
 \[g\in C^{(r)}_{i-1}(x)h^{-1},\text{ }ghx\in U_i \text{ and }\abs{F_{n_r}g\cap F_{n_r}C^{(r)}_{i-1}(hx)}<\eps\abs{F_{n_r}}\]
 \[gh\in C^{(r)}_{i-1}(x),\text{ }ghx\in U_i \text{ and }\abs{F_{n_r}g\cap F_{n_r}C^{(r)}_{i-1}(x)h^{-1}}<\eps\abs{F_{n_r}}\]
 \[gh\in C^{(r)}_{i-1}(x),\text{ }ghx\in U_i \text{ and }\abs{F_{n_r}gh\cap F_{n_r}C^{(r)}_{i-1}(x)}<\eps\abs{F_{n_r}}\]
 \[gh\in C^{(r)}_i(x)\]
 \[g\in C^{(r)}_i(x)h^{-1}\]
 Also note that since the set $C^{(r)}_i(x)$ is determined by the visits of $x$ in clopen sets of $X$ under the action of $G$, for every $F\in G$ there exists an $\eta$ such that if $d(x,y)<\eta$, then $C^{(r)}_i(x)\cap F=C^{(r)}_i(y)\cap F$. Now, let $C^{(r)}(x)=C^{(r)}_m(x)$ and let $\CT^{(r)}_x$ be a point in $\{\emptyset, F_{n_1},\dots,F_{n_r}\}^G$ such that $\CT^{(r)}_x)_g=F_{n_r}$ if $g\in C^{(r)}(x)$ and $\emptyset$ otherwise. By the remarks made earlier, the mapping $x\mapsto \CT^{(r)}_x$ is continuous, and also $(\CT^{(r)}_{hx})_g=F_{n_r}$ if and only if $g\in C^{(r)}(hx)$, if and only if $g\in C^{(r)}(x)h^{-1}$ if and only if $gh\in C^{(r)}(x)$, if and only if $(\CT^{(r)}_x)_{gh}=F_{n_r}$, which means that $\CT^{(r)}_{hx}=h(\CT^{(r)}_x)$, and thus the mapping $x\mapsto \CT^{(r)}_x$ commutes with the shift action.
 
 We will now show that for every $x\in X$, the family $\set{F_{n_r}c:c\in C^{(r)}(x)}$ is an $\eps$-covering quasitiling of $G$ which is $\eps$-disjoint.
 
 As for the $\eps$-disjointness, note that for every $i$ and every $c\in C^{(r)}_i(x)$ the set $F_{n_r}c\setminus F_{n_r}C^{(r)}_{i-1}(x)$ is a $(1-\eps)$ subset of $F_{n_r}c$. Such subsets of $F_{n_r}c_1$ and $F_{n_r}c_2$ are pairwise disjoint by definition if $c_1$ and $c_2$ belong to $C^{(r)}_i(x)$ for different $i$'s, and if they both belong to the same $C^{(r)}_i(x)$ then if $F_{n_r}c_1\cap F_{n_r}c_2$ is nonempty, then it has an element of the form $t_1c_1=t_2c_2$ for $t_1,t_2\in F_{n_r}$ and $c_1,c_2\in C^{(r)}_i(x)$. Thus $t_1c_1(x)=t_2c_2(x)$, but as $c_1(x)$ and $c_2(x)$ are both in $U_i$, this implies that the images of $U_i$ under $t_1$ and $t_2$ are not disjoint, which contradicts the definition of $U_i$.
 
 Finally, for every $x\in X$ and every $g\in G$ there exists an $i$ such that $gx\in U_i$. Now either $g\in C^{(r)}_i(x)$, or $\abs{F_{n_r}g\cap F_{n_r}C^{(r)}(x)}\geq \abs{F_{n_r}g\cap F_{n_r}C^{(r)}_{i-1}(x)}\geq \eps\abs{F_{n_r}}$ --- in both cases, $F_{n_r}g\cap F_{n_r}C^{(r)}(x)\geq \eps\abs{F_{n_r}}$, which means that the lower Banach density of $F_{n_r}C^{(r)}(x)$ is at least $\eps$.

Fix some $j\in\{1,2,\dots,r-1\}$ and suppose that for every $x\in X$ we have constructed an $\eps$-disjoint quasitiling $\CT_x^{(j+1)}=\{F_{n_i}c:j+1\le i\le r,\ c\in C^{(j)}(x)\}$, such that for every $x$ the union
$$
H^{(j+1)}(x) = \bigcup_{i=j+1}^r\ F_{n_i}C^{(i)}(x)
$$
has lower Banach density strictly larger than $1-(1-\frac\eps2)^{r-j}$ (this is our inductive hypothesis on $H^{(j+1)}(x)$ and it is fulfilled for $H^{(r)}(x)$), and that $x\mapsto \CT^{(j+1)}_x$ is a factor map. We need to go one step further in our ``decreasing induction'', i.e., add a center set $C^{(j)}(x)$ for the shape $F_{n_j}$.

As before, we can cover $X$ by a finite family of clopen sets $U_1,U_2,\ldots,U_m$, such that for every $i$ the sets $g(U_i)$ are pairwise disjoint for $g\in F_{n_j}$. Let $C^{(j)}_0(x)=\emptyset$, and for $i=1,\ldots,m$, let

\[\begin{split}
&C_i^{(j)}(x)=C_{i-1}^{(j)}(x)\cup\\
&\cup\set{g\in G: gx\in U_i \text{, and } \abs{F_{n_j}g\cap \left(F_{n_j}C^{(j)}_{i-1}(x)\cup \bigcup_{l={j+1}}^kF_{n_l}C^{(l)}(x)\right)}<\eps\abs{F_{n_j}}}.
\end{split}\]

Let $C^{(j)}(x)=C^{(j)}_m(x)$. By the same arguments as before, we easily obtain the following properties of this set:
\begin{itemize}
    \item For any $x\in X$ and $h\in G$ we have  $C^{(j)}(hx)=C^{(j)}(x)h^{-1}$.
    \item The mapping $x\mapsto \CT^{(j)}_x$ is a factor map.
    \item For every $x\in X$, the quasitiling $\CT^{(j)}x$ is  $\eps$-disjoint.
    \item For every $x$, if we denote by $H^{(j)}(x)$ the union of the above family, then for every $g\in G$ we have 
    \begin{equation}\label{eq:Hjdensity}\frac{\abs{H^{(j)}(x)\cap F_{n_j}g}}{\abs{F_{n_j}}}>\eps.\end{equation}
\end{itemize}
The rest of the proof is nearly identical as the analogous proof in \cite{DHZ} for ``static'' quasitilings. We copy it here for completeness, adapting it to our dynamical situation. Since the arguments below involve only algebraic operations and combinatorics (and not the dynamics and topology on $X$), and are the same for every $x$, for the next few paragraphs we will omit references to $x$ and write just $H^{(j)}$ rather than $H^{(j)}(x)$ and $C^{(j)}$ rather than $C^{(j)}(x)$, implicitly stating that the estimates provided are true for every $x$.

Our goal is to estimate from below the lower Banach density of $H^{(j)}$. By Lemma \ref{bd}, it suffices to estimate $\underline D_F(H^{(j)})$ for just one finite set $F$ which we will define in a moment. Define $B=\Bigl(\bigcup_{i=j+1}^rF^2_{n_i}\Bigr)F_{n_j}$. Clearly, $B$ contains $F_{n_j}$
(hence the unit), and, as easily verified, it has the following property:

\begin{equation}\begin{split}
&\text{    Whenever $F_{n_j}F_{n_i}c\cap A\neq\emptyset$, for some
    $i\in\{j+1,\dots,r\}$,}\\&\text{ $c\in G$ and $A\subset G$, then $F_{n_i}c\subset BA$.}
    \end{split}
\end{equation}

Let $n$ be so large that $F_n$ is $(B,\delta_j)$-invariant and that $\underline D_{F_n}(H^{(j+1)})>1-(1-\frac\eps2)^{r-j}$ (the latter is possible due to the assumption on $\underline D(H^{(j+1)})$). Now we define the aforementioned set $F$ as $F=F_{n_j}F_n$.
%Since $(F_{n_j}F_n)_{n\ge 1}$ is also a F\o lner sequence, we can additionally require $n$ to be so large that $\underline D_{F_{n_j}F_n}(H^{(j+1)})\ge \frac1{1+\delta_j}\underline D(H^{(j+1)})$.

Fix some $g\in G$ and define
$$
\alpha_g = \frac{|H^{(j+1)}\cap F_ng|}{|F_n|}  \text{ \ \ and \ \ } \beta_g = \frac{|H^{(j+1)}\cap BF_ng|}{|BF_n|}.
$$
Notice that
\begin{equation}\label{minus}
\alpha_g \ge \underline D_{F_n}(H^{(j+1)}) > 1-(1-\tfrac\eps2)^{r-j}.
\end{equation}
Also, we have
\begin{align}
\beta_g &\ge \frac{|H^{(j+1)}\cap F_ng|}{(1+\delta_j)|F_n|}=\frac{\alpha_g}{1+\delta_j}\label{zero} \ ,   \text{ \ \ and }\\
\beta_g &\le \frac{|H^{(j+1)}\cap F_ng|+|BF_ng\setminus F_ng|}{|F_n|}\le \alpha_g + \delta_j.
\end{align}
Note that since $F_{n_j}\subset B$ and $F_n$ is $(B,\delta_j)$-invariant, $F_n$ is automatically
$(F_{n_j},\delta_j)$-invariant. Thus
\begin{equation}\label{jeden}
\frac{|H^{(j+1)}\cap Fg|}{|F|}\ge \frac{|H^{(j+1)}\cap F_ng|}{(1+\delta_j)|F_n|}= \frac{\alpha_g}{1+\delta_j}\ge\frac{\beta_g-\delta_j}{1+\delta_j}.
\end{equation}

Consider only these finitely many component sets $F_{n_i}c$ of $H^{(j+1)}$ (i.e., with $i\in\{j+1,\dots,r\},\ c\in C^{(i)}$) for which $F_{n_j}F_{n_i}c$ has a nonempty intersection with $F_ng$, and denote by $E_g$ the union of so selected components $F_{n_i}c$. By the property (*) of $B$ (with $A=F_ng$), $E_g$ is a subset of $BF_ng$ (and also of $H^{(j+1)}$), so
\begin{equation}\label{niewiem}
|E_g|\le |H^{(j+1)}\cap BF_ng|= \beta_g|BF_n|\le\beta_g(1+\delta_j)|F_n|.
\end{equation}
Each of the selected components $F_{n_i}c\subset E_g$ is $(F_{n_j},\delta_j)$-invariant, hence, when multiplied on the left by $F_{n_j}$ it can gain at most $\delta_j|F_{n_i}c|$ new elements. Thus the set $E_g$, when multiplied on the left by $F_{n_j}$, can gain at most $\delta_j\sum_{F_{n_i}\!c\subset E_g}|F_{n_i}c|$ new elements. On the other hand, denoting by $(F_{n_i}c)^\circ$ the pairwise disjoint sets (contained in respective sets $F_{n_i}c$) as in the definition of $\eps$-disjointness, we also have
$$
\sum_{F_{n_i}\!c\subset E_g}|F_{n_i}c|\le \frac1{1-\eps}\sum_{F_{n_i}\!c\subset E_g}|(F_{n_i}c)^\circ| = \frac1{1-\eps}\Bigl|\bigcup_{F_{n_i}\!c\subset E_g}(F_{n_i}c)^\circ\Bigr|\le \frac1{1-\eps}|E_g|.
$$
Combining this with the preceding statement, we obtain that the set $E_g$, when multiplied on the left by $F_{n_j}$, can gain at most $\frac{\delta_j}{1-\eps}|E_g|$ new elements, which is less than $2\delta_j|E_g|$ (we can assume that $\eps<\frac12$).
Denote $\hat{H}^{(j+1)} = F_{n_j}H^{(j+1)}$. By the choice of the components included in $E_g$, the set $F_{n_j}E_g$ contains all of $\hat{H}^{(j+1)}\cap F_ng$. Thus, using $(1+2\delta_j)\le (1+\delta_j)^2$ and \eqref{niewiem}, we obtain that
\begin{equation*}
|\hat{H}^{(j+1)}\cap F_ng|\le |F_{n_j}E_g| \le (1+2\delta_j)|E_g| \le (1+\delta_j)^3\beta_g|F_n|.
\end{equation*}
Let $N_g = F_ng\setminus \hat{H}^{(j+1)}$. By the above inequality, we know that
\begin{equation}\label{dwa}
|N_g|\ge \bigl(1-(1+\delta_j)^3\beta_g\bigr)|F_n|\ge \bigl(1-(1+\delta_j)^3\beta_g\bigr)\tfrac{|F|}{1+\delta_j},
\end{equation}
where the last inequality follows from the $(F_{n_j},\delta_j)$-invariance of $F_n$.

Earlier we have established (inequality (\ref{eq:Hjdensity})) that $\frac{\abs{H^{(j)}\cap F_{n_j}c}}{\abs{F_{n_j}}}\geq\eps$ for every $c \in G$, in particular for every $c \in N_g$. This implies that there are at least $\eps|N_g||F_{n_j}|$ pairs $(f,c)$ with $f\in F_{n_j}, c\in N_g$ such that $fc\in H^{(j)}$. This in turn implies that there exists at least one $f\in F_{n_j}$ for which
\begin{equation}\label{trzy}
|H^{(j)}\cap fN_g|\ge \eps|N_g|.
\end{equation}

Notice that $fN_g$ is contained in $Fg$ (because $N_g\subset F_ng$ and $f\in F_{n_j}$) and disjoint from $H^{(j+1)}$ ($N_g$ is disjoint from $\hat{H}^{(j+1)}$ which contains $f^{-1}H^{(j+1)}$). Thus we can estimate, using \eqref{jeden}, \eqref{dwa} and \eqref{trzy}:
\begin{eqnarray*}
\frac{|H^{(j)}\cap Fg|}{|F|}&\ge & \frac{|H^{(j+1)}\cap Fg| + |H^{(j)}\cap fN_g|}{|F|}\\
&= & \frac{|H^{(j+1)}\cap Fg|}{|F|}  + \frac{|H^{(j)}\cap fN_g|}{|N_g|}\frac{|N_g|}{|F|}\\
&\ge & \frac{\beta_g-\delta_j}{1+\delta_j} + \eps\frac{1-(1+\delta_j)^3\beta_g}{1+\delta_j}.
\end{eqnarray*}

%Recall that $\underline D_{F_{n_j}F_n}(H^{(j+1)})\ge \frac1{1+\delta_j}\underline D(H^{(j+1)})$
%Recall our goal: lower estimate of $\underline D_F(H^{(j)})$. We must estimate from below the ratio $\frac{|H^{(j)}\cap Fg|}{|F|}$ for any $g\in G$.
%Given $g$, l

Both terms in the last expression are linear functions of $\beta_g$, the first one
with positive and large slope $\frac1{1+\delta_j}$, the other with negative but small slope $-\eps(1+\delta_j)^2$. Jointly,
the function increases with $\beta_g$. So, we can replace $\beta_g$ by any smaller value, for instance, by $\frac{1-(1-\frac\eps2)^{r-j}}{1+\delta_j}$ (see \eqref{minus} and \eqref{zero}), to obtain
$$
\frac{|H^{(j)}\cap Fg|}{|F|}> \frac{1-(1-\tfrac\eps2)^{r-j}}{(1+\delta_j)^2}-\frac{\delta_j}{1+\delta_j} + \eps\bigl(\tfrac1{1+\delta_j}-(1+\delta_j)(1-(1-\tfrac\eps2)^{r-j})\bigr).
$$
Now notice, that if we replace the undivided occurrence of $\eps$ by $\frac{3\eps}4$, we make the entire expression smaller by
some positive value (independent of $g$). On the other hand, if $\delta_j$ is very small and we remove it completely from the
expression, we will perhaps enlarge it, but very little. We now specify $\delta_j$ to be so small, that if we replace $\eps$ by $\frac{3\eps}4$ and remove $\delta_j$ completely, then the expression will become smaller. With such a choice of $\delta_j$ we have
$$
\frac{|H^{(j)}\cap Fg|}{|F|}> 1-(1-\tfrac\eps2)^{r-j} +\frac{3\eps}4(1-\tfrac\eps2)^{r-j} = 1-(1-\tfrac\eps2)^{r-j+1} + \xi,
$$
where $\xi>0$ does not depend on $g$. Taking infimum over all $g\in G$ we get, by Lemma \ref{bd},
$$
\underline D(H^{(j)})\ge \underline D_F(H^{(j)}) > 1-(1-\tfrac\eps2)^{r-j+1},
$$
and the inductive hypothesis has been derived for $j$.

Once the induction reaches $j=1$ we get that the lower Banach density of $H^{(1)}(x)$ is larger than $1-(1-\tfrac\eps2)^r$ which,
by the choice of $r$, is larger than $1-\eps$ and means that $\CT_x=\CT^{(1)}(x)$ is the desired quasitiling.
\end{proof}

At the cost of increasing the number of possible shapes (but without sacrificing the other properties) we can make our quasitilings disjoint:

\begin{cor}\label{disjoint_tilings}
Let $X$ be a compact, zero-dimensional metric space and let $G$ be a countable amenable group acting freely on $X$ by homeomorphisms, with a F\o lner sequence $(F_n)$ of symmetric sets containing the unit. Given $\eps>0$ and any positive integer $n_0$, there exists a dynamical quastitiling $x\mapsto \CT_x$ which is \emph{disjoint}, and $(1-\eps)$-covering, and such that every shape $S$ of every $\CT_x$ is a $(1-\eps)$-subset of some F\o lner set $F_{n(S)}$ where $n(S)>n_0$.
\end{cor}

\begin{proof}
For every $x$, let $\hat{\CT}_x$ be the quasitiling delivered by Lemma \ref{quasitilings}, with $\eps$ and $n_0$. Recall that for every $x$, the set of shapes of $\CT_x$ is $\set{F_{n_1},\ldots,F_{n_r}}$, where $n_0<n_1<\ldots<n_r$. Furthermore, every tile of $\hat{\CT}_x$ has the form $\hat{T}=F_{n_j}c$ for some $j\in\set{1,\ldots,r}$ and $c\in C^{(j)}_i(x)$. Let $T=F_{n_j}c\setminus\left(F_{n_j}C^{(j)}_{i-1}(x)\cup \bigcup_{l={j+1}}^kF_{n_l}C^{(l)}(x)\right)$. By the definition of the sets $C^{(j)}_i(x)$ and $C^{(j)}(x)$, $T$ is a $(1-\eps)$-subset of $\hat T$. In addition, if $\hat{T}\neq \hat{T'}$, then $T$ and $T'$ are disjoint: If we represent $\hat{T}=F_{n_j}c$ and $\hat{T'}=F_{n_{j'}}c'$, then we have two possibilities: either $j=j'$, and $c$ and $c'$ belong to the same $C^{(j)}_i(x)$ --- in this case $\hat{T}$ and $\hat{T'}$ are disjoint, therefore so are $T$ and $T'$ as their respective subsets. Otherwise we can without loss of generality assume that either $j<j'$, or $j=j'$ and $i>i'$ --- in this case, since $T=F_{n_j}c\setminus\left(F_{n_j}C^{(j)}_{i-1}(x)\cup \bigcup_{l={j+1}}^kF_{n_l}C^{(l)}(x)\right)$, it is disjoint from $T'$, as $T'$ is a subset of $F_{n_{j'}}c'$ and thus is included in the set subtracted from $F_{n_j}c$. 

Let $\CT_x$ denote the quasitiling obtained from $\hat{\CT}_x$ by these modifications. Note that due to the properties of the sets $C^{(j)}_i(x)$, we still have $\CT_{gx}=g(\CT_x)$, and if we  set $F=\bigcup_{j=1}^{r}F_{n_j}$, then  if $\hat{\CT_x}$ agrees with $\hat{\CT_y}$ on a subset of $G$ of the form $FB$, then $\CT_x$ and $\CT_y$ agree on $B$. This means that $\CT_x$ depends on $x$ continuously, and thus the mapping $x\mapsto \CT_x$ is a factor map. In addition, for every $x$ the union of all the tiles of $\CT_x$ has not changed, so the new quasitiling is still $(1-\eps)$-covering.

%
%Enumerate the elements of the set $F=\bigcup_{j=1}^{r}F_{n_j}$ by $g_1,g_2,\ldots,g_N$, with $g_1=e$ (the unit of $G$). Now, every tile of $\CT_x$ has the form $F_{n_j}c$ for some $j\leq r$, $c\in C^{(j)}(x)$, so we can enumerate its elements using the formula $g\mapsto n$ if and only if $g=g_nc$. Let $E_{x,c}c$ be the set of all elements of $F{n_j}c$ which do not belong to other tiles, or if they do, they have larger numbers in those tiles than in $F_{n_j}c$. Since $g_1=e$, we are guaranteed that $e\in E_{x,c}$. Furthermore, since $\CT_x$ was $\eps$-disjoint, we have $\abs{E_{x,c}}>(1-\eps)\abs{F_{n_j}}$. As there are only finitely many possible $E_{x,c}$'s for various $x$ and $c$ (they are subsets of finitely many finite sets), they form the shapes of a new quasitiling (and every $C^{(j)}(x)$ becomes split into multiple center sets corresponding to the various $E_{x,c}$. That way for every $x$ we obtain a disjoint, quasitiling $\CT_x$, and since the union of the new tiles is precisely the same as the union of the old tiles, this new quasitiling is also $(1-\eps)$-covering. Furthermore, observe that the procedure of making the tiles disjoint is invariant under right translation, and if $\hat{\CT_x}$ agrees with $\hat{\CT_y}$ on a subset of $G$ of the form $FB$, then $\CT_x$ and $\CT_y$ agree on $B$, which means that $\CT_x$ depends on $x$ continuously, and thus the mapping $x\mapsto \CT_x$ is a factor map.
\end{proof}

Finally, we can also obtain a quasitiling that is ``compatible'' with another quasitiling by smaller tiles:

\begin{lem}\label{congruent_tilings}
    Let $G$ be an amenable group acting freely on a zero-dimensional metric space $X$ and let $x\mapsto \CT_x$ be any disjoint dynamical quasitiling of $G$. For any $\eps>0$, any finite $K\subset G$ and any $\delta>0$ there exists a disjoint, 
    $(1-\eps)$-covering dynamical quasitiling $x\mapsto \CT'_x$ such that every shape of $\CT'_x$ is $(K,\delta)$-invariant, and every tile of $\CT_x$ is either a subset of some tile of $\CT'_x$ or is disjoint from all such tiles.
\end{lem}
\begin{proof}
    First of all observe that there exist constants $\delta'$ and $\eta$ such that if $T$ is $(K,\delta')$-invariant then any set $T'$ such that $\abs{T'\sd T}<\eta\abs{T}$, is $(K,\delta)$-invariant. We can also assume that $(1-\frac{\eps}{2})(1-\eta)>1-\eps$. Let $U$ be the union of all shapes of $\CT_x$ over all $x\in X$; there exists a $\delta''$ such that if $T$ is $(U,\delta'')$-invariant, and we denote by $T_U$ the set $\set{t\in T: Ut\subset T}$, then $\abs{UT\setminus T_U}<\eta\abs{T}$. Once these parameters are set, lemma \ref{quasitilings} ensures the existence of a disjoint, dynamical, $(1-\frac{\eps}{2})$-covering quasitiling $x\mapsto \CT''x$ whose shapes are $(E,\delta'')$ invariant for all shapes $E$ of $\CT_x$. 
    
%    As before, enumerate the elements of $U$ by $g_1,g_2,\ldots,g_N$, with $g_1=e$. This induces an enumeration on every shape $E$ of $\CT_x$. If $x\in X$, and every $T$ is a tile of $\CT_x$, then $T$ has the form $Ec$, which induces an enumeration of $T$ as in the previous proof: $t_1=g_1c, t_2=g_2c$ and so on (if for some $i$ the element $t_ic$ is not in $T$, we just skip that number). Let $i(T)$ be the smallest index for which $g_ic\in T$ and $g_ic\in T''$ for some $T''\in \CT''_x$. 

Every tile $T\in \CT_x$ has a unique representation in the form $Ec$, where $E$ is one of the finitely many shapes. If $c$ belongs to some tile $T''$ of $\hat{\CT''}_x$ (by disjointness of the quasitiling, there can be at most one such tile), let $\phi_x(T)$ be such a $T''$. Otherwise let  $\phi_x(T)=\emptyset$. This gives us a mapping from $\CT_x$ to $\CT''_x\cup \emptyset$, which by its construction is continuous and commutes with the dynamics. Now we can modify $\CT''_x$ as follows: If $T''\in \CT''_x$, add to $T''$ all the tiles $T$ of $\CT_x$ such that $T''=\phi_x(T)$, and remove from it all the $T$ of $\CT_x$ such that $T''\neq \phi_x(T)$. 

    Observe that if $T'$ denotes a tile obtained from $T''$ after all such modifications, then $T''_U\subset T'\subset UT''$. It follows that $\abs{T''\sd T'}<\eta\abs{T''}$, and thus $T'$ is $(K,\delta)$-invariant. Therefore the map $x\mapsto\CT'_x$ obtained as a result of these modification is a disjoint, dynamical quasitiling whose tiles are all $(K,\delta)$-invariant and every tile of $\CT_x$ is either a subset of some tile of $\CT'_x$ or disjoint from all such tiles. Finally, to estimate the lower Banach density of the union of all tiles of $\CT'_x$, observe that for every tile $T''$ of $\CT''_x$ there exists a tile $T'$ of $\CT'_x$ which contains a $(1-\eta)$ subset of $T''$. Since $\CT''_x$ is a $(1-\frac{\eps}{2})$-covering quasitiling, a straightforward argument (see e.g. lemma 3.4 of \cite{DHZ}) implies that the lower Banach density of the union of all tiles of $\CT'_x$ is at least $(1-\frac{\eps}{2})(1-\eta)>1-\eps$.
    
\end{proof}

As a final note, we remark that replacing our quasitilings with actual tilings (in which the union of tiles is all of $G$) remains an open problem. If the procedure of completing a disjoint quasitiling to a tiling used in \cite{DHZ} was applied individually to $\CT_x$ for every $x\in X$, there is no guarantee that such new tilings would still depend continuously on $x$, and it is currently unclear whether the technique can be suitably modified.

% Non-BibTeX users please use


\begin{thebibliography}{}
%
% and use \bibitem to create references. Consult the Instructions
% for authors for reference list style.
%
%\bibitem[BBF]{BBF}
%Mathias Beiglb{\"o}ck, Vitaly Bergelson, and Alexander Fish, Sumset phenomenon in countable amenable groups, Adv. Math. 223 no.~2, 416--432 (2010). MR 2565535 (2011a:11011)
%\bibitem[BD]{BD}
%Mike Boyle and Tomasz Downarowicz, The entropy theory of symbolic
%  extensions, Invent. Math. 156, no.~1, 119--161 (2004). MR 2047659
%  (2005d:37015)
%
%\bibitem[Da]{Da}
%Alexandre~I. Danilenko, Entropy theory from the orbital point of view,
%  Monatsh. Math. 134, no.~2, 121--141  (2001). MR 1878075 (2002j:37011)
%
%\bibitem[Do]{Do}
%Tomasz Downarowicz, Entropy in dynamical systems, New Mathematical
%  Monographs, vol.~18, Cambridge University Press, Cambridge (2011). MR 2809170(2012k:37001)

%\bibitem[DFR]{DFR}
%Tomasz Downarowicz, Bartosz Frej and Pierre-Paul Romagnoli,
%Shearer's inequality and Infimum Rule for Shannon entropy and topological entropy,
%preprint, arXiv:1502.07459
%
%{\color{red}
%\bibitem[DH]{DH}
%Tomasz Downarowicz and Dawid Huczek, \emph{Faithful zero-dimensional principal
%  extensions}, Studia Math. 212 (2012), no.~1, 1--19. 3004163
%}

%{\color{blue} This reference is not cited in the paper.}

%{\color{red}
%\bibitem[H]{H} Dawid Huczek, \emph{Zero-dimensional extensions of amenable
%group actions}, (preprint)

%{\color{blue} Please update the state of this paper.}

\bibitem[DHZ]{DHZ}
T. Downarowicz, D. Huczek, G. Zhang, Tilings of Amenable Groups, preprint (arXiv: 1502.02413)

%
%\bibitem[L]{L}
%Elon Lindenstrauss, Pointwise theorems for amenable groups, Invent.
%  Math. 146, no.~2, 259--295 (2001). MR 1865397 (2002h:37005)
%
%\bibitem[LW]{LW}
%Elon Lindenstrauss and Benjamin Weiss, Mean topological dimension,
%  Israel J. Math. 115, 1--24 (2000). MR 1749670 (2000m:37018)
%
%\bibitem[MO]{MO}
%Jean Moulin~Ollagnier, Ergodic theory and statistical mechanics, Lecture
%  Notes in Mathematics, vol. 1115, Springer-Verlag, Berlin (1985). MR 781932 (86h:28013)

\bibitem[N]{N}
I.~Namioka, F\o lner's conditions for amenable semi-groups, Math. Scand.
  15, 18--28 (1964). MR 0180832 (31 \#5062)

%\bibitem[OW1]{OW1}
%Donald~S. Ornstein and Benjamin Weiss, Ergodic theory of amenable group
%  actions. {I}. {T}he {R}ohlin lemma, Bull. Amer. Math. Soc. (N.S.) 2, no.~1, 161--164 (1980). MR 551753 (80j:28031)

\bibitem[OW]{OW}
Donald~S. Ornstein and Benjamin Weiss, Entropy and isomorphism theorems for actions of amenable groups, J. Analyse Math. 48, 1--141 (1987). MR 910005 (88j:28014)

%\bibitem[RW]{RW}
%Daniel~J. Rudolph and Benjamin Weiss, \emph{Entropy and mixing for amenable
%  group actions}, Ann. of Math. (2) 151, no.~3, 1119--1150  (2000).
%  MR 1779565 (2001g:37001)

%\bibitem[We]{We}
%Benjamin Weiss, Monotileable amenable groups, Topology, ergodic theory,
%  real algebraic geometry, Amer. Math. Soc. Transl. Ser. 2, vol. 202, Amer.
%  Math. Soc., Providence, RI, pp.~257--262 (2001). MR 1819193 (2001m:22014)

%\bibitem[WZ]{WZ}
%Thomas Ward and Qing Zhang, The {A}bramov-{R}okhlin entropy addition
%  formula for amenable group actions, Monatsh. Math. 114,
%  no.~3-4, 317--329 (1992). MR 1203977 (93m:28023)


% \bibitem{K}
%  U.Krengel,
%  \emph{Ergodic Theorems},
%  Walter de Gruyter, Berlin, New York (1985).

%\bibitem{Pa}
%	K.R. Parthasarathy, \emph{Probability Measures On Metric Spaces}, Academic Press, New York (1967).

%\bibitem{P}
%  R.R.Phelps,
%  \emph{Lectures on Choquet's Theorem},
%  D. Van Nostrand Company, Princeton, New Jersey (1966).

\end{thebibliography}
\end{document}